\documentclass[12pt,reqno]{amsart}
\usepackage{amssymb}
\usepackage{tikz-cd}

\numberwithin{equation}{section}
\theoremstyle{plain}
\newtheorem{teo}[equation]{Theorem}
\newtheorem{cor}[equation]{Corollary}

\newtheorem{prop}[equation]{Proposition}
\newtheorem{note}[equation]{Notation}

\theoremstyle{remark}

\theoremstyle{plain}
\newtheorem{defn}[equation]{Definition}
\theoremstyle{plain}
\newtheorem{pb}[equation]{Problem}

\newcommand{\mb}[1]{\mathbb{#1}}
\newcommand{\mc}[1]{\mathcal{#1}}
\newcommand{\mr}[1]{\mathrm{#1}}
\newcommand{\ov}[1]{\overline{#1}}

\DeclareMathOperator{\Ima}{Im}

\DeclareMathOperator{\K3}{K3}
\DeclareMathOperator{\P1}{\mathbb{P}^1}
\DeclareMathOperator{\gon}{gon}
\DeclareMathOperator{\Mg}{\mathcal{M}_g}

\DeclareMathOperator{\mo}{\mc{O}}
\DeclareMathOperator{\N}{\mc{N}}
\DeclareMathOperator{\T}{\mc{T}}
\DeclareMathOperator{\mult}{\mathrm{mult}}

\newtheorem*{reftheorem*}{Theorem \reftotheorem}
\newenvironment{reftheorem}[1]
  {\newcommand{\reftotheorem}{\ref{#1}}\begin{reftheorem*}}
  {\end{reftheorem*}}

\newtheorem*{refcor*}{Corollary \reftocor}

\begin{document}

\title[Notes on solvability of curves on surfaces]
{Solvability of curves on surfaces}

\author[A. Dan]{Ananyo Dan}

\address{BCAM - Basque Centre for Applied Mathematics, Alameda de Mazarredo 14,
48009 Bilbao, Spain}

\email{adan@bcamath.org}

\author[M. Z. Fashami]{ Mohamad  Zaman Fashami}

\address{K. N. Toosi University of technology, Tehran Province, Tehran, Mirdamad Blvd, No. 470, Iran}

\email{zamanfashami65@yahoo.com}

\author[N. Zangani]{Natascia Zangani}

\address{Scuola di dottorato in Matematica,
Dottorando,
Via Sommarive, 14 - 38123 Povo, Italy}

\email{n.zangani@unitn.it}



\date{\today}

\begin{abstract}
In this article, we study subloci of solvable curves in $\mc{M}_g$ which are contained in either a $\K3$-surface or a 
quadric or a cubic surface. We give a bound on the dimension of such subloci. In the case of complete intersection genus $g$ curves in a cubic surface, we show that a general such curve is solvable.
\end{abstract}
\maketitle

\section{Introduction}
We consider the following classical question about solvability, which Enriques stated as unsolved in 1897 during the Congress of Mathematicians in Zurich.
\begin{pb}
Given a complex curve $C$, we denote by $\widetilde{K(C)}$ the Galois closure of $K(C)$. Is there a curve $D$ and a covering $\pi \colon C\to D$ such that the field extension $K(D)\hookrightarrow \widetilde{K(C)}$ is solvable?
\end{pb}
When considering this problem, we restrict ourselves to considering the case $D=\P1$. Given a covering $\pi\colon C\twoheadrightarrow \P1$, we can consider the Galois group of the splitting field of the extension $K(C)\colon \mathbb{C}(x)$, where $K(C)$ is the function field of the curve and $\mathbb{C}(x)$ is the rational function field over $\mathbb{C}$. In particular, this Galois group is the monodromy group of the covering $M(\pi)$(see \cite[Proposition pp. 189]{harris1979}).

For high genera we have the fundamental Zariski's Theorem (see \cite{zariski1926sull}).
\begin{teo}[Zariski's Theorem]
Let $C$ be a very general smooth complex projective curve of genus $g\ge 7$, then for every $\pi\colon C\twoheadrightarrow \P1$ surjective of degree $d>0$, the monodromy group $M(f)$ is not solvable.
\end{teo}
If we consider lower genera
$g\le 6$, then $gon(C)\le\lfloor\frac{g+3}{2}\rfloor\le 4$, hence there exists a covering $\pi\colon C\to \P1$ of degree $4$ and the monodromy group is a subgroup of $S_4$, thus it is solvable.
By Zariski's Theorem we have that if $g\ge 7$, then the sublocus $\mc{M}_{g,\mathrm{sol}}$ defined by the solvable curves in the moduli space $\mc{M}_{g}$, is different from the whole space. We are interested into studying this sublocus.

In Section \ref{section:third} we apply Zariski's argument to estimate the codimension of $\mc{M}_{g,\mathrm{sol}}\subseteq \mc{M}_{g}$. Since a general covering factors as a primitive covering $C\to \P1$ and another covering $\P1\to \P1$, we reduce ourselves to the study of primitive solvable covering, which we denote by PS--covering or PS--curves.

In Section \ref{section:fourth}  we focus on curves lying on a $K3$ surface and we prove that a general element is not a PS--curve. We denote as $\ov{\mc{M}_{g,\mr{sol}}^{K3}\backslash \mc{M}_{g,4}}$ (closure taken in $\mc{M}_g$) the sublocus of $\mc{M}_g$ parametrizing PS--curves lying on a $K3$ surface which are not four--gonal curves. Applying the Zariski's argument, we give an estimate on its codimension:
\begin{reftheorem}{teo:K3} 
For $g \ge 7$, a general element of $\mc{M}_g^{K3}$ is not a PS--curve. Furthermore, for a maximal dimensional irreducible component $L$ of $\ov{\mc{M}_{g,\mr{sol}}^{K3}\backslash \mc{M}_{g,4}}$ (closure taken in $\mc{M}_g$), we have 
 \begin{enumerate}
  \item if $7 \le g \le 9$ and $g=11$, the codimension of $L$ in $\mc{M}_g^{K3}$ is at least $7$;
  \item if $g=10$ then the codimension of $L$ in $\mc{M}_g^{K3}$ is at least $12$;
  \item if $g=12$ then the codimension of $L$ in $\mc{M}_g^{K3}$ is at least $14$;
  \item if $g=11$ or $g \ge 13$ then the codimension of $L$ in $\mc{M}_g^{K3}$ is at least $15$.
 \end{enumerate}
\end{reftheorem}

In the last section we apply Zariski's argument to study the subloci of complete intersection, primitive, solvable curves lying on  cubic and  quadric surfaces.
Denote by $g(a,b)$ the genus of a complete intersection curve in $\mb{P}^3$ obtained by the intersection of two general surfaces of degree $a$ and $b$, then it holds (\cite[Remark IV.$6.4.1$]{hartshorne2013algebraic}):
\[
g(a,b)=\frac{1}{2}ab(a+b-4)+1.
\]
Denote by $\mc{M}_{g(a,b)}^a$ the subloci in $\mc{M}_{g(a,b)}$ of genus $g(a,b)$ curves contained in a degree $a$ hypersurface in $\mb{P}^3$. 
Denote by $\mc{M}_{g(a,b),sol}^a$ the sublocus of $\mc{M}_{g(a,b)}^a$ parametrizing PS--curves.

\begin{reftheorem}{teo:quadric}
The codimension $c(a,b)$ of an irreducible component of $\mc{M}_{g(a,b),sol}^a$ in $\mc{M}_{g(a,b)}^a$ satisfies the following:
\begin{enumerate}
\item If $a=2, b \ge 5$ then $c(a,b)>0$,
\item If $a=3, 3 \le b \le 5$ then $c(a,b)>0$,
\item If $a=2, 2 \le b \le 4$ then $c(a,b)=0$.
\end{enumerate}
\end{reftheorem}

\vspace{0.2 cm}
\emph{Acknowledgements}
We thank G. P. Pirola and L. Stoppino for their help and support during the Pragmatic $2016$. We thank A. Knutsen for his interest and suggestions.
We also thank the organizers of Pragmatic $2016$ for an inspiring atmosphere of research.

\section{Notation and preliminaries}
We work on the complex field $k=\mathbb{C}$, all curves are considered to be complex smooth projective curves.

We say that a property $\mathcal{P}$ is \emph{very general} if it holds in $\Mg\backslash \bigcup_{n\in\mathbb{N}} Z_n$, where $Z_n$ is a proper closed Zariski subset of $\Mg$ for any $n$. We call a curve very general if there exists $\{Z_n\}_{n\in\mathbb{N}}$ as above such that $\left[C\right]\notin \bigcup_{n\in \mathbb{N}}Z_n$.

Given a curve $C$, the \emph{gonality} of $C$ is defined as
\begin{align*}
\gon(C)&:=\min\{d\in \mathbb{Z}_{>0} \,\colon\, \exists\pi\colon C\to \P1 \text{ s.t. } \deg(\pi)=d\}\\
&\,=\min\{d\colon \, \exists g^1_d \text{ over } C\}.
\end{align*}
The notion of gonality gives a measure on how far is the curve from being rational. 
We denote as $\mathcal{M}_{g,k}$ the subspace defined by the curves $C$ such that $\gon(C)\le k$. 
The locus $\mathcal{M}_{g,k}$ is known to be an irreducible subvariety of dimension $2g+2k-5$ if $k\le \frac{g+2}{2}$, and if $k\ge \lfloor\frac{g+3}{2}\rfloor$, then $\mathcal{M}_{g,k}=\Mg$. Thus $\lfloor\frac{g+3}{2}\rfloor$ is called the \emph{general gonality}. 

We define $k$--gonal locus as $X_k:=\mathcal{M}_{g,k}\backslash \mathcal{M}_{g,k-1}$, $X_k$ is a quasi-projective subvariety of $\mathcal M_g$ whose points correspond to $k$-gonal curves.
We have that 
\[
\dim \left(X_k\right)=\min\{2g-5+2k,3g-3\}
\]
for $k\le \frac{g+2}{2}$, otherwise $\mathcal{M}_{g,k}=\Mg$ and $\dim \left(X_k\right)=0$.

The gonality gives a stratification of the moduli spaces of curve $\Mg$ for $g\ge3$:
\[
\mathcal{M}_{g,2}\subseteq \mathcal{M}_{g,3} \subseteq \dots \subseteq \mathcal{M}_{g,k} \subseteq \dots \subseteq \Mg.
\]
For further details we refer the reader to \cite{arbarello1981footnotes}.

Given a covering $\pi\colon C\to \P1$, we denote by $\{b_1,\dots, b_k\}$ the branch locus and by $F$ the generic fiber.
We denote by $M(\pi)$ the monodromy group of the covering, i.e. the image of the monodromy map
\begin{align*}
\mu_\pi\colon \pi_1\left(\mb{P}^1\backslash\{b_1,\dots,b_k\}\right)&\longrightarrow \mathrm{Aut}(F)\cong S_d\\
[\gamma]&\quad \mapsto\quad \mu_\pi \left([\gamma]\right)
\end{align*}
where $\mu_{\pi}([\gamma])(p_i)=\widetilde{\gamma_{p_i}}(1)$ is defined by the lifted path $\widetilde{\gamma_{p_i}}$ with base point $p_i$.
We recall that a group $G$ is \emph{solvable} if it admits a finite filtration of subgroups
\[
\{1\}=G_t\subseteq \dots \subseteq G_{i+1}\subseteq G_i\subseteq \dots G_1=G,
\]
such that $G_{i+1}\triangleleft G_i$ is a normal subgroup and $G_i\backslash G_{i+1}$ is abelian  for any $i$.
We call a covering $\pi$ undecomposable if it does not factor nontrivially; this is equivalent to ask that $M(\pi)$ is primitive, i.e. it does not have any block (see \cite[Proposition 1.2.10]{rotapirola}).
We say that a curve $C$ is \emph{solvable}, resp. \emph{primitive}, if it admits a covering $\pi\colon C\to \P1$ such that $M(\pi)$ is solvable, resp. primitive. We call a primitive and solvable curve a PS--curve.
\begin{defn}
We denote by $\mc{M}_{g,\mathrm{sol}}$ the sublocus of $\Mg$ defined by the curves that admit a cover to $\P1$ with solvable monodromy group.
\end{defn}

Given a covering $\pi\colon C\to \P1$ of degree $d$ with $l$ branch points, we consider its monodromy group $M(\pi)$ that is a subgroup of $S_d$. We can choose $l$ conjugacy classes in $S_d$, $c_1,\dots c_l$ such that we have a natural map $b_i\mapsto c_i$.
We can consider the \emph{Hurwitz space}\footnote{For a nice and detailed introduction on Hurwitz spaces we refer to \cite{romagny2006hurwitz}}
$H(c_1,\dots,c_l,d)$ which parametrizes the covering $\pi\colon C\to \P1$ such that $M(\pi)$ is of type $(c_1,\dots, c_l)$ and $ \deg(\pi)=d$. 
We can define the \emph{configuration space} $X_l$ as
\[
X_l:=\left[\left(\P1\times \dots \times\P1\right)-\bigcup_{i\neq j}\Delta_{ij}\right]/{\mathrm{Sym}_l},
\]
where $\Delta_{ij}=\{(z_1,\dots,z_l)\in \P1 \times \dots \times \P1 \,\colon \, z_i=z_j\}$. If the Hurwitz space is nonempty, then we have a surjective map given by associating to each covering its branch divisor:
\begin{align*}
H(c_1,\dots, c_l,d)&\xrightarrow{h} X_l\\
\pi\, &\mapsto\, (b_1+\dots+b_l).
\end{align*}
We have the following diagram
\[
\begin{tikzcd}
H(c_1,\dots, c_l,d) \arrow[r,"\mu"]\arrow[d,"h"]
& \mc{M}_g\\
X_l
\end{tikzcd}
\]
where $\mu$ sends a cover $\pi\colon C\to \P1$ to the corresponding point $[C]\in \Mg$.

\section{Zariski's argument}

Here we briefly review the Zariski's argument, for further details we refer the reader to \cite{zariski1926sull} and \cite{pirola2005curve}.
First of all, we recall some preliminary result about solvable groups.
\begin{prop}[Proposition 2.1,\cite{pirola2005curve}]
\label{prop:Z1}
Let $G\subseteq S_d$ be a primitive solvable subgroup acting on a set $X$, and let $x\in X$. Then
\begin{enumerate}
\item there exists a unique minimal normal subgroup $A\triangleleft G$;
\item $A$ is an elementary abelian $p$--group for some $p$ prime;
\item $G=AG_x$ and $A\cap G=1_G$;
\item $A$ acts regularly on $X$, i.e. for any $x\in X$ the map $a\mapsto ax$ gives a bijection from $A$ to $X$.
\end{enumerate}
In particular $d=|A|=p^k$ for a prime number $p$.
\end{prop}
\begin{prop}[\cite{zariski1926sull}]
\label{prop:Z2}
Let $d=p^k$ with $p$ prime and let $G$ be a primitive solvable subgroup of $S_d$ acting on a set $X$. Then for any $g\in G$
\[|X^g|=|\{x\in X\, \colon \, gx=x\quad \forall x\in X\}|\le p^{k-1}.
\]
\end{prop}
The Zariski's argument is a count of moduli obtained by applying the classical Riemann-Hurwitz formula (see \cite[IV, Corallary 2.4]{hartshorne2013algebraic}). If we consider a $d\colon 1$ covering $\pi\colon C\to \P1$, by means of the Riemann-Hurwitz formula we get
\begin{equation}
2g(C)-2=-2d+r,
\end{equation}
where $d=\deg(\pi)$ and $r$ is the degree of the ramification divisor $R_\pi=\sum_{p\in X}\left[\mathrm{mult}_p(\pi)-1\right]\cdot p$. 
We recall that the \emph{branch divisor} is defined as $\pi_*(R_\pi)$. We denote by $b(y)$ the multiplicity of a branch point $y\in \P1$, i.e.
\[
b(y)=\sum_{p\in \pi^{-1}(y)}\left[\mult_p(\pi)-1\right].
\]
The following bound on $b(y)$ is the key result for the Zariski's argument and the proof if it is given for the reader's convenience.
\begin{prop}[Zariski's argument \cite{zariski1926sull}]
\label{prop:ZA}
Let $\pi\colon X\to Y$ be a $d\colon 1$ primitive and solvable covering of curves. Then, there exists a prime $p$ such that $d=p^k$ for some integer $k$ and for any branch point $y\in Y$ the following holds
\begin{equation}
\label{eq:ZA}
b(y)\ge \frac{p^k-p^{k-1}}{2}.
\end{equation}
\end{prop}
\begin{proof}
Let $M(\pi)$ be the monodromy group of $\pi$ and let us consider the generic fiber $F$.  There is a natural  action of $M(\pi)$ defined on the fiber: $M(\pi)\curvearrowright F$. We denote by $\mr{orb}(g)$ the number of orbits of $<g>$ for $g\in M(\pi)$. By induction on $d$ one can prove that
\[
b(y)=\sum_{p\in\pi^{-1}(y)}\left[\mult_p(\pi)-1\right]=d-\mathrm{orb}(g).
\]
We denote by $n$ the number of fixed points of $g$. Since $b(y)\geq \mr{orb}(g)-n$, we have
\begin{gather*}
2b(y)\ge b(y)+\mr{orb}(g)-n=d-n\\
\Rightarrow \, b(y)\ge \frac{d-n}{2}.
\end{gather*}
Since $M(\pi)$ is assumed to be primitive, then by Proposition \ref{prop:Z1} and Proposition \ref{prop:Z2}: $d=p^k$ and $n\le p^{k-1}$.
\end{proof}

\section{Solvable locus in moduli of curves: the general case}
\label{section:third}
We denote as $\mc{M}_{g,\mathrm{sol}}$ the sublocus of $\Mg$ defined by the solvable curves of genus $g$.
In 1991 Michael G. Neubauer proved that $\mc{M}_{g, \mathrm{sol}}$ is not dense in $\Mg$ in the interesting case $g>6$ (see \cite{neubauer1992monodromy} ).

\begin{teo}[{\cite[Theorem 1.11]{neubauer1992monodromy}}]\label{s8}
Let $g>6$. Then $\mc{M}_{g,\mathrm{sol}}$ is a quasi--projective subvariety of $\Mg$ with strictly positive codimension.
\end{teo}

We can give an estimate on the codimension of $\mc{M}_{g,\mathrm{sol}}$ by applying the Zariski's argument (Proposition \ref{prop:ZA}).
\begin{prop}\label{s1}
 Let $g \ge 7$.
 Let $\mc{Z} \subset \mc{M}_g$ be an irreducible family of smooth projective curves of genus $g$ such that the general element in the family 
 admits a $d:1$ solvable and primitive covering of $\mb{P}^1$ with $d \ge 5$. Then, $\dim \mc{Z} \le g+4$.
\end{prop}

\begin{proof}
 We proceed by contradiction. Suppose that $\dim \mc{Z}>g+4$. Take a general curve $C \in \mc{Z}$ and let $\pi:C \to \mb{P}^1$ be a $d:1$ branched covering for some $d \ge 5,$
 such that the corresponding monodromy group is primitive and solvable. By Proposition \ref{prop:Z1}, this implies $d=p^k$ for some positive integer $k$ and $p \ge 2$.
 By Proposition \ref{prop:ZA}, for any branched point $y \in \mb{P}^1$, it holds $b(y) \ge (p^k-p^{k-1})/2$. We denote by $l$ the number of branch points of $\pi$ and we consider $X_l$, the $l$-dimensional configuration space, which is a covering of an open subscheme of $(\mb{P}^1)^{l}$.
 Since $C$ is chosen to be general, there exists a Hurwitz space $H(c_1, \dots, c_l,d)$, parametrizing
$d:1$ covers of $\mb{P}^1$ by genus $g$ curves with monodromy group of type $(c_1, \dots, c_l)$ such that $T:=\mu(H(c_1, \dots, c_l,d))$
contains $\mc{Z}$ (here $\mu$ is the forgetful map from $H(c_1, \dots, c_l,d)$ to $\mc{M}_g$). 

 As the map from $H(c_1, \dots, c_l,d)$ to the configuration space is generically finite, \[\dim H(c_1, \dots, c_l,d)=l.\]
Since the group of automorphism of $\mb{P}^1$ is $3$-dimensional, $\dim T=l-3$. By the Riemann-Hurwitz formula and by Zariski's argument \eqref{eq:ZA}, 
\[
2g-2 \ge -2p^k+l \cdot\frac{p^k-p^{k-1}}{2}.
\] 
By assumption, 
 $l-3 =\dim T \geq \dim Z> g+4$, we have
 \[
 2\cdot\frac{2g-2+2p^k}{p^k-p^{k-1}} - 3>g+4.
 \]
Since $d=p^k\ge 5$, we get $p^{k-1}(p-1) \ge 4$. Applying this to the previous inequality, we have
 \[g-1+\frac{4p}{p-1}-3 \ge 2\frac{2g-2}{p^k-p^{k-1}}+\frac{4p^k}{p^k-p^{k-1}} - 3>g+4\]
 which is equivalent to $(4p)/(p-1)>8$. This implies $p<2$, which contradicts the assumption $p \ge 2$. 
 \end{proof}

\section{Curves on K3 surfaces}
\label{section:fourth}
We consider the moduli of genus $g$ curves contained in polarized K3 surfaces and we study the sub-locus of solvable curves. Let us first recall the definition of a K3 surface and some useful properties.
\begin{defn}
 A $K3$ \emph{surface} is a complete non-singular surface $X$ such that the canonical sheaf $\mc{K}_X$ is trivial and $H^1(\mo_X)=0$. 
 A \emph{polarization on }
  $X$ is an ample, primitive\footnote{By primitive, we mean that $\mc{L}$ is not the power of any other invertible sheaf.} invertible sheaf $\mc{L}$ on $X$. 
  The pair 
 $(X,\mc{L})$ is said to be a \emph{canonically polarized K3 surface of genus} $g$ if $\mc{L}$ is a polarization on $X$ satisfying 
 \[
 \mc{L}^2=2g-2,
 \]
 where by $\mc{L}^2$ we mean the self-intersection of any curve $C$ in the linear system defined by $\mc{L}$.
\end{defn}

\begin{defn}
 We say that two polarized K3 surfaces $(X,\mc{L})$ and $(Y,\mc{L'})$ are isomorphic if there exists $\phi\colon X \to Y$ isomorphism of schemes satisfying 
 $\phi^*\mc{L}=\mc{L}'$.
 \end{defn}
 The \emph{moduli functor}  
 \[
 \widetilde{\kappa_g}\colon\mr{Schemes}/\mb{C} \longrightarrow \{\mr{Sets}\},
 \] 
is defined as the functor which associates (up to isomorphism) to a $\mb{C}$-scheme $S$ the set of pairs 
 $((\pi\colon\mc{X}_S \to S), \mc{L})$ where 
 \begin{enumerate}
\item $\pi$ is a smooth, proper morphism;
\item $\mc{L}$ is an invertible sheaf on $\mc{X}_S$ such that, for every geometric point $s \in S$, the fiber 
 $(\mc{X}_s,\mc{L}|_{\mc{X}_s})$ is a canonically polarized K3 surface of genus $g$.
\end{enumerate}
 
 Similarly, we can define the functor 
 \[
 \widetilde{\mc{P}_g}\colon\mr{Schemes}/\mb{C} \longrightarrow \{\mr{Sets}\},
 \] 
 which associates to a $\mb{C}$-scheme $S$ (up to isomorphism) the set of triples 
 $((\pi'\colon\mc{C}_S \to S), (\pi\colon\mc{X}_S \to S), \,\mc{L})$ where
 \begin{enumerate}
\item  $\pi', \pi$ are smooth, proper morphisms;
\item  $\mc{C}_S \subset \mc{X}_S$;
\item $\mc{L}$ is an invertible sheaf on $\mc{X}_S$ such that, for every geometric point $s \in S$, the fiber $(\mc{X}_s,\mc{L}|_{\mc{X}_s})$ is a canonically 
 polarized K3 surface of genus $g$ and $\mc{C}_s \in |\mc{L}|_{\mc{X}_s}|$. 
\end{enumerate}

\begin{teo}[\cite{cili}]\label{not1}
 There exists a coarse moduli spaces $\mc{P}_g$ and $\kappa_g$ corepresenting the moduli functors $\widetilde{\mc{P}_g}$ and $\widetilde{\kappa_g}$.
 The natural projection map induces a $\mb{P}^g$-bundle structure on $\mc{P}_g \to \kappa_g$. Moreover, $\dim(\kappa_g)=19$ and $\dim(\mc{P}_g)=19+g$.
\end{teo}

Morally, $\mc{P}^g$ parametrizes the pairs $(C,S)$ where $C$ is  s smooth projective curve, and $S$ is a polarized $K3$ surface containing $C$.

\begin{teo}[\cite{cili}]\label{not2}
 The natural forgetful map $\phi_g\colon\mc{P}_g \to \mc{M}_g$ is dominant if and only if $2 \le g \le 9$ and $g=11$. Moreover, $\phi_g$ is generically finite if and only if 
 $g=11$ and $g \ge 13$. For $g=10$ the map $\phi_g$ has fiber dimension $3$ and $\phi_{12}$ has fiber dimension $1$.
\end{teo}

Using the formula for the dimension of $\mc{P}_g$ in Theorem \ref{not1} and that of the fiber of the forgetful map $\phi_g$ as in Theorem \ref{not2}, we can directly prove that:
\begin{cor}\label{s2}
 The following hold:
 \begin{enumerate}
  \item if $2 \le g \le 9$ 
  then $\dim \Ima \phi_g=3g-3$;
  \item if $g=10$ then $\dim \Ima \phi_g=16+g$;
  \item if $g=12$ then $\dim \Ima \phi_g=18+g$;
  \item if $g=11$ or $g \ge 13$ then $\dim \Ima \phi_g=19+g$.
 \end{enumerate}
\end{cor}

\begin{defn}
We denote by $\mc{M}_g^{K3}$ the image of $\phi_g$ and by $\mc{M}_{g,\mr{sol}}^{K3}$ the sublocus of $\mc{M}_g^{K3}$ parametrizing PS--curves.
\end{defn}

\begin{defn}
Given positive integers $g,r,d$, the \emph{Brill-Noether number}, is 
\[
\rho(g,r,d)=g-(r+1)(g-d+r).
\]
Given a curve $C \in \mc{M}_g$, denote by $W^r_d(C) \subset \mc{M}_g$ the space of all degree $d$ invertible sheaves $\mc{L}$ satisfying $h^0(\mc{L}) \ge r+1$. 
\end{defn}

\begin{teo}
\label{teo:K3}
 For $g \ge 7$, a general element of $\mc{M}_g^{K3}$ is not a PS--curve. Furthermore, for a maximal dimensional irreducible component $L$ of $\ov{\mc{M}_{g,\mr{sol}}^{K3}\backslash \mc{M}_{g,4}}$ (closure taken in $\mc{M}_g$), we have 
 \begin{enumerate}
  \item if $7 \le g \le 9$ and $g=11$, the codimension of $L$ in $\mc{M}_g^{K3}$ is at least $7$;
  \item if $g=10$ then the codimension of $L$ in $\mc{M}_g^{K3}$ is at least $12$;
  \item if $g=12$ then the codimension of $L$ in $\mc{M}_g^{K3}$ is at least $14$;
  \item if $g=11$ or $g \ge 13$ then the codimension of $L$ in $\mc{M}_g^{K3}$ is at least $15$.
 \end{enumerate}

\end{teo}

\begin{proof}
By \cite{lazarsfeld1986brill} if $\rho(g,1,d)<0$ then for a general curve $C$ of genus $g$ contained in a $K3$ surface, $W^r_d(C)=\emptyset$, so in particular, there does not exist 
any $d:1$ covering from $C$ to $\mb{P}^1$. For $g \ge 7$, $\rho(g,1,d)<0$ 
if and only if $d \le 4$. For $d \le 4$, a $d$-gonal curve in solvable, so we want to exclude the sublocus $\mc{M}_{g,4}$ in which the maximal gonality is reached by the four--gonals curves. Then \[\ov{\mc{M}_{g,\mr{sol}}^{K3} \cap \mc{M}_{g,4}}=\ov{\mc{M}_g^{K3} \cap \mc{M}_{g,4}} \not= \mc{M}_g^{K3},\] with closure taken in $\mc{M}_g$. Let $L$ be an irreducible component of $\ov{\mc{M}_{g,\mr{sol}}^{K3} \backslash \mc{M}_{g,4}}$. Since a general element in $L$ is solvable, $d$-gonal for $d \ge 5$, by Proposition \ref{s1}, $\dim L \le g+4$. Using Corollary \ref{s2} for $g \ge 7$, observe 
\[\dim \Ima \phi_g=\dim \mc{M}_g^{K3}>g+4.\] By Theorem \ref{s8}, there 
are finitely many irreducible components of $\mc{M}_{g,\mr{sol}}^{K3}$. 
Since there are finitely many irreducible components of $\mc{M}_{g,\mr{sol}}^{K3}$ and every component is of dimension strictly less than that of $\mc{M}_g^{K3}$, a general element of $\mc{M}_g^{K3}$ is not solvable. This proves the first part of the theorem.

The second part of the theorem follows directly using Proposition \ref{s1} and Corollary \ref{s2}. This completes the proof of the theorem.
 \end{proof}


%

\section{Curves on quadric and cubic surfaces}
In this section we study the subloci of solvable curves contained in quadric or cubic surfaces. The first step is to compute the fiber dimension of the moduli map (see Proposition \ref{s4}). In order to compute the codimension of the solvable subloci of the above mentioned curves, we need to use Proposition \ref{s1}. To do so, we need to know the gonality of such curves. This is done in Proposition \ref{sol5}. We combine these steps in Theorem \ref{teo:quadric} to compute the required codimension.

\begin{note}
Given a Hilbert polynomial $P$ of a curve $C$ in $\mb{P}^3$, denote by $\mr{Hilb}_P$ the Hilbert scheme parametrizing all subschemes in $\mb{P}^3$ with Hilbert polynomial $P$. Let $a, b$ be positive integers and $b \ge a$. Denote by $g(a,b)$ (resp. $P(a,b)$) the genus (resp. Hilbert polynomial) of a complete intersection curve in $\mb{P}^3$ obtained by the intersection of a general surface of degree $a$ and another of degree $b$.
\end{note}

\begin{prop}\label{s4}
 Suppose $a \ge 2$ and $b \ge 3$. Let $X$ be a smooth, projective surface in $\mb{P}^3$ of degree $a$ and $C$ be 
the complete intersection of $X$ with a general degree $b$ surface in $\mb{P}^3$. Then, the dimension of the fiber over $[C] \in \mc{M}_{g(a,b)}$ of the moduli map 
$\mu\colon\mr{Hilb}_{P(a,b)} \to \mc{M}_{g(a,b)}$ is at most $h^0(\T_{\mb{P}^3})$, where 
$\T_{\mb{P}^3}$ is the tangent sheaf on $\mb{P}^3$.
\end{prop}

\begin{proof}
Using deformation  theory observe that the differential to the moduli map 
\[d\mu\colon  T_{[C]}\mr{Hilb}_{P(a,b)} \to T_{[C]}\mc{M}_{g(a,b)}\] is the boundary map $H^0(\N_{C|\mb{P}^3}) \to H^1(\T_C)$ coming from the short exact sequence:
\[0 \to \T_{C} \to \T_{\mb{P}^3} \otimes \mo_{C} \to \N_{C|\mb{P}^3} \to 0.\]
Using the genus formula for complete intersection curves (see \cite[Remark IV.$6.4.1$]{hartshorne2013algebraic}) one can check that $g(a,b)>1$, and this implies $\deg \T_C=2-2g(a,b)<0$.
This means that the kernel of $d\mu$ is isomorphic to $H^0(\T_{\mb{P}^3} \otimes \mo_{C})$. Since $T_{[C]}\mu^{-1}([C])$ is isomorphic to $\ker d\mu$, it suffices to prove that $H^0(\T_{\mb{P}^3} \otimes \mo_{C}) \cong H^0(\T_{\mb{P}^3})$.

Consider now the following Koszul complex associated to the curve $C$:
\[0 \to \mo_{\mb{P}^3}(-a-b) \to  \mo_{\mb{P}^3}(-a) \oplus  \mo_{\mb{P}^3}(-b) \to  \mo_{\mb{P}^3} \to \mo_C \to 0.  \]
Tensoring by $\T_{\mb{P}^3}$, we get the exact sequence:
\begin{equation}\label{sol1}
0 \to \T_{\mb{P}^3}(-a-b) \to  \T_{\mb{P}^3}(-a) \oplus  \T_{\mb{P}^3}(-b) \to  \T_{\mb{P}^3} \to  \T_{\mb{P}^3}|_C \to 0.
\end{equation}
Choose coordinates $X, Y, Z, W$ for $\mb{P}^3$ i.e., $\mb{P}^3=\mr{Proj} \mb{C}[X,Y,Z,W]$. Recall, the twisted Euler sequence:
\[0 \to \Omega^1_{\mb{P}^3}(t) \to \mo_{\mb{P}^3}(t-1)^{\oplus 4} \xrightarrow{\theta} \mo_{\mb{P}^3}(t) \to 0\]
where $\theta$ is defined by $(P_1, P_2, P_3, P_4)$ maps to $(P_1X+P_2Y+P_3Z+P_4W)$ for $P_i \in \Gamma(\mo_{\mb{P}^3}(t-1))$ for $t \ge 1$. 
Hence, \[H^0(\theta)\colon H^0(\mo_{\mb{P}^3}(t-1)^{\oplus 4}) \to H^0(\mo_{\mb{P}^3}(t))\] is surjective for all $t \ge 1$. 
Observe, $H^1(\mo_{\mb{P}^3}(t))=0=H^2(\mo_{\mb{P}^3}(t))$ for all $t \in \mb{Z}$ and $H^3(\mo_{\mb{P}^3}(t))=0$ for all $t>-4$.
Hence, $H^1( \Omega^1_{\mb{P}^3}(t))=0$ for all $t \ge 1$, $H^2( \Omega^1_{\mb{P}^3}(t))=0$ for all $t \in \mb{Z}$ and $H^3( \Omega^1_{\mb{P}^3}(t))=0$ for all $t>-3$. Dualizing, we have $H^2(\T_{\mb{P}^3}(-t-4))=0$ for all $t \ge 1$, $H^1(\T_{\mb{P}^3}(-t-4))=0$ for all $t \in \mb{Z}$ and $H^0(\T_{\mb{P}^3}(-t-4))=0$ for all $t>-3$. Expanding (\ref{sol1}), we get the following exact sequences:
\begin{eqnarray}
& 0 \to  \T_{\mb{P}^3}(-a-b) \to  \T_{\mb{P}^3}(-a) \oplus  \T_{\mb{P}^3}(-b) \to M_1 \to 0 \label{sol2}\\
& 0 \to M_1 \to \T_{\mb{P}^3} \to  \T_{\mb{P}^3}|_C \to 0 \label{sol3}
\end{eqnarray}
The long exact sequence associated to (\ref{sol2}) implies 
\[
H^0(M_1)=0=H^1(M_1).
\]
Applying this to the long exact sequence associated to (\ref{sol3}) we get $H^0(\T_{\mb{P}^3} \otimes \mo_{C}) \cong H^0(\T_{\mb{P}^3})$. 
\end{proof}

\begin{defn}
Denote by $\mc{M}_{g(a,b)}^a$ the subloci in $\mc{M}_{g(a,b)}$ of genus $g(a,b)$ curves contained in a degree $a$ hypersurface in $\mb{P}^3$. 
Denote by $\mc{M}_{g(a,b),sol}^a$ the sublocus of $\mc{M}_{g(a,b)}^a$ parametrizing PS--curves.
\end{defn}

We recall a standard construction of a cubic surface obtained by blowing up $6$ points. This description will be used 
to compute the gonality of a curve in a cubic surface (see Proposition \ref{sol5}).

\begin{defn}\label{sol6}
Let $X_3$ be a smooth cubic surface and $\pi\colon X_3 \to \mb{P}^2$ the blow-up of $\mb{P}^2$ at six points $p_1,...,p_6$ on the plane, 
no three collinear and not all six lying on a conic. Denote by $E_1,...,E_6$ the exceptional curves in $X_3$ over $p_1,...,p_6$, respectively.
Let $l' \subset \mb{P}^2$ be a line and $l:=\pi^*([l'])$.
\end{defn}

\begin{prop}\label{sol5}
Let $a \ge 2$ and $[C] \in \mc{M}_{g(a,b)}^a$. Then, the gonality $\mr{gon}(C)$ satisfies:
\begin{enumerate}
\item If $a=2$ then $\mr{gon}(C)=b$,
\item If $a=3$ then $\mr{gon}(C)=2b$.
\end{enumerate}
\end{prop}

\begin{proof}
By \cite[Theorem $4.2$]{bas}, 
$\mr{gon}(C)=ab-l$ where $l$ is the maximum number of points of $C$ on a line. 

(1): Suppose $a=2$. 
By definition, \[C \sim \left(\sum\limits_{i=1}^b p_i \times \mb{P}^1\right)+\left(\sum\limits_{i=1}^b \mb{P}^1 \times p_b\right) \mbox{ for any set of } b \mbox{ points } p_1,...,p_b.\]
As $X \cong \mb{P}^1 \times \mb{P}^1$, any line on $X$ is of the form $\mb{P}^1 \times \{x\}$ or $\{x\} \times \mb{P}^1$
for $x \in \mb{P}^1$. Since $(\mb{P}^1 \times \{x\})\cdot(\mb{P}^1 \times \{y\})=0=(\{x\} \times \mb{P}^1)\cdot(\{y\} \times \mb{P}^1)$ and
$(\{x\} \times \mb{P}^1)\cdot(\mb{P}^1 \times \{y\})=1$ for any $x,y \in \mb{P}^1$, $C\cdot l=b$ for any line $l \subset X$. For any line $l$ not contained in $X$, $C\cdot l \le X\cdot l = 2$. Hence, the gonality $\mr{gon}(C)=\min\{2b-b, 2b-2, 2b-1,2b\}=b$. This proves $(1)$.

(2): Suppose $a=3$. Recall, $X_3$ (see Definition \ref{sol6}) contains $27$ lines, $E_i$ for $i=1,...,6$, $F_{ij}$, $i \not= j$ and $G_j$ for $j=1,...,6$ where $F_{ij} \sim L-E_i-E_j$ and $G_j \sim 2L-\sum_{i \not= j} E_i$ (see \cite[Proposition V.$4.8$ and Theorem V.$4.9$]{hartshorne2013algebraic}). By \cite[Proposition V.$4.8$]{hartshorne2013algebraic},
the hyperplane section $H$ is linearly equivalent to $3L-\sum E_i$. Since $C \sim bH$, we have $C\cdot E_i=b, C\cdot L=3b, C\cdot F_{ij}=3b-2b=b$ and 
$C\cdot G_j=6b-5b=b$. Hence, the gonality of $C$ is $2b$. 
\end{proof}

\begin{teo}
\label{teo:quadric}
The codimension $c(a,b)$ of an irreducible component of $\mc{M}_{g(a,b),sol}^a$ in $\mc{M}_{g(a,b)}^a$ satisfies the following:
\begin{enumerate}
\item If $a=2, b \ge 5$ then $c(a,b)>0$,
\item If $a=3, 3 \le b \le 5$ then $c(a,b)>0$,
\item If $a=2, 2 \le b \le 4$ then $c(a,b)=0$.
\end{enumerate}
\end{teo}

\begin{proof}
 By Proposition \ref{s4}, the dimension of the fiber of the moduli map $\mu:\mr{Hilb}_{P(a,b)} \to \mc{M}_{g(a,b)}$ is at most $h^0(\T_{\mb{P}^3})$. It is easy to check that \[\dim \mr{Hilb}_{P(a,b)}=\binom{a+3}{3}+\binom{b+3}{3}-\binom{b-a+3}{3}-2-h^0(\mo_{\mb{P}^3}(a-b)). \]
 So, \[\dim \Ima \mu \ge \frac{a^3+11a}{3}+\frac{b^2a-ba^2}{2}+2ba-16-h^0(\mo_{\mb{P}^3}(a-b)). \]
 By Proposition \ref{sol5}, the gonality of $C \in \mc{M}_{g(a,b)}^a$ is at least $5$ in the case $a=3, b \ge a$ and when $a=2, b \ge 5$.
For these values of $a$ and $b$, by Proposition \ref{s1}, the codimension $c(a,b)$ of $\mc{M}_{g(a,b),sol}^a$  is at least \[\frac{a^3+11a}{3}+\frac{2b^2a-2ba^2-b^2a^2}{4}+3ba-21-h^0(\mo_{\mb{P}^3}(a-b)).\] 

(1): Substituting $a=2$ in the above equation, we observe that for $b \ge 5$, $c(2,b)>0$.

(2): Substituting $a=3$ in the above equation, we observe \[c(3,b) \ge \frac{18b-3b^2}{4}-1.\] For $3 \le b \le 5$, we have $c(3,b)>0$. 

(3): Substituting $a=2$, by Proposition \ref{sol5}, the gonality of $C\in \mc{M}_{g(a,b)}^a$ is strictly less than $5$ for  $2 \le b \le 4$. 
Hence, $C$ is solvable. This implies that $\mc{M}_{g(a,b),sol}^a$ coincides with $\mc{M}_{g(a,b)}^a$. 
\end{proof}

\newpage
\nocite{*}
\bibliographystyle{alpha}
\bibliography{bibliography}

\end{document}